\documentclass{article}

\usepackage[cp1250]{inputenc}
\usepackage{ifthen}
\usepackage[affil-it]{authblk}

\usepackage{arydshln}
\usepackage{multirow}

\usepackage{enumerate}
\usepackage{geometry}
\usepackage{amsmath}
\usepackage{amsthm}
\usepackage{amsfonts}
\usepackage{amssymb}
\usepackage{makeidx}
\usepackage{color}
\usepackage{epsfig}
\usepackage{floatflt}
\usepackage{graphics,color}
\usepackage{graphicx,color}
\usepackage{polski}
\usepackage[polish, english]{babel}
\usepackage{setspace}
\usepackage{fancyhdr}
\usepackage{indentfirst}
\usepackage{endnotes}
\usepackage{latexsym}
\usepackage{tikz}
\usepackage{tkz-graph}
\usepackage[all]{xy}

\usepackage{lmodern}

\newtheorem{tw}{Theorem}

\theoremstyle{remark}
\newtheorem{uw}{Remark}

\theoremstyle{definition}
\newtheorem{ex}{Example}
\newtheorem{df}{Definition}

\begin{document}

\title{Reproducing Kernel Hilbert Space Associated with a Unitary Representation of a Groupoid}
\author{Monika Drewnik${}^{1}$\thanks{Corresponding author: monikadrewnik@gmail.com}}
\author{Tomasz Miller${}^{2}$}
\author{Zbigniew Pasternak-Winiarski${}^{3}$}

\affil{\small ${}^1$ College of Rehabilitation, Department of Rehabilitation, Kasprzaka 49, 01-234 Warsaw, Poland}
\affil{\small ${}^2$ Copernicus Center for Interdisciplinary Studies, Jagiellonian University,
\\Szczepa\'nska 1/5, 31-011 Krak\'ow, Poland}
\affil{\small ${}^3$ Faculty of Natural and Health Sciences, The John Paul II Catholic University of Lublin,\\ Konstantyn\'{o}w 1 H, 20-708 Lublin, Poland}

\maketitle

\begin{abstract}
The aim of the paper is to create a link between the theory of reproducing kernel Hilbert spaces (RKHS) and the notion of a unitary representation of a group or of a groupoid. More specifically, it is demonstrated on one hand, how to construct a positive definite kernel and an RKHS for a given unitary representation of a group(oid), and on the other hand how to retrieve the unitary representation of a group or a groupoid from a positive definite kernel defined on that group(oid) with the help of the Moore--Aronszajn theorem. The kernel constructed from the group(oid) representation is inspired by the kernel defined in terms of the convolution of functions on a locally compact group. Several illustrative examples of reproducing kernels related with unitary representations of groupoids are discussed in detail. The paper is concluded with the brief overview of the possible applications of the proposed constructions.
\end{abstract}

\textit{Keywords:} Reproducing Kernel Hilbert Spaces, unitary representation, groupoid, group, Haar measure, convolution

\noindent

\section{Introduction}
\label{sec:Introduction}

The theory of reproducing kernel Hilbert spaces (RKHS) provides research tools in such domains as complex analysis, probability theory and statistics \cite{ber}, stochastic (Gaussian) processes \cite{kal}, quantum physics \cite{odz,pas} or computer science (especially artificial intelligence \cite{dre,gar}). Basic properties and definitions together with the detailed analysis of RKHS can  be found in \cite{paulsen,alp}. In \cite{bar} RKHS associated with the continuous wavelet transform generated by the irreducible representations of the Euclidean Motion SE(2) are considered. 

Groupoids are widely used in differential geometry, algebraic topology and physics. There is a well-known association of groupoids with $C^\ast$-algebras \cite{dele,ren}. The definition and main properties of (locally compact) groupoids can be found for example in \cite{da,pat}. In \cite{pys1,pys2} unitary representations of groupoids are considered. The applications of theory of finite groudoids and their representations are presented in \cite{ibo1,ibo}. Considering  representations of locally compact groups, the Haar measure and the tensor product of the Hilbert spaces, we refer the Reader to \cite{waw}.

The aim of te present paper is to connect the above two domains. To this end, we study the relationship between unitary representations of groupoids and reproducing kernel Hilbert spaces, proposing how to construct one using the other provided certain conditions are met.

The content of the paper is as follows. In Section \ref{sec:Reproducing Kernel Hilbert Space} the fundamental definitions and properties of the theory of reproducing kernel Hilbert spaces are introduced. Section \ref{sec:Groupoids} covers the concept of groupoids (propositions, examples and the notion of their unitary representation). Main results of the paper are contained in Section \ref{sec:Results}. It combines the ideas presented in the two previous sections. The constructions of reproducing kernels associated to a unitary representation of a group and to a unitary representation of a groupoid are described in Subsections \ref{subsec:Results Groups} and \ref{subsec:Results Groupoids}, respectively. The latter subsection is concluded with several illustrative examples. Section \ref{sec:Applications} contains brief discussion on the possible applications of the notions studied in the article.

\section{Reproducing Kernel Hilbert Space}
\label{sec:Reproducing Kernel Hilbert Space}

\subsection{Basic concepts and properties}
\label{subsec:Podstawowe pojecia}

\begin{df}
Let $A$ be a nonempty set. The map $ K: A \times A \rightarrow \mathbb{C}$ shall be called a \textit{kernel} on $A$. We say that the kernel $K$ is \textit{positive definite} if
\begin{align*}
\forall_{ n \in \mathbb{N}} \quad \forall_{ a_1,...,a_n \in A } \quad \forall_{ \lambda_1,...,\lambda_n \in  \mathbb C} \qquad \sum^{n}_{k=1} \sum^{n}_{l=1}\lambda_k K(a_k, a_l) \bar{\lambda}_l \geq 0 .
\end{align*}
\end{df}

Let $H$ be an inner product space of complex-valued functions on $A$ equipped with the inner product\footnote{We adopt the convention that inner products are anti-linear with respect to the second argument.} $\left\langle \cdot,\cdot \right\rangle$.

\begin{df}
The family $\left\{ K_a \right\}_{a \in A} \subset H$ is called a \textit{reproducing family} of $H$ if
\begin{align}
\forall_{ a \in A} \quad \forall_{ f \in H} \qquad f(a)= \left\langle f,K_a \right\rangle.
\label{reprodukcja}
\end{align}
Equality (\ref{reprodukcja}) is called \textit{the reproducing kernel property}, and the function $K: A \times A \rightarrow \mathbb{C}$ defined as
\begin{align*}
K(b,a):= \langle K_a, K_b \rangle = K_{a}(b), \quad a,b \in A
\end{align*}
is called a \textit{reproducing kernel} on the space $H$. If the latter is a Hilbert space, we call $H$ a \emph{reproducing kernel Hilbert space} (RKHS).
\end{df}
Clearly, by the Riesz representation theorem, any RKHS has a unique reproducing family and thus a unique reproducing kernel, which can be easily shown to be positive definite. The following seminal result shows that actually the converse is also true: Every positive definite kernel is a reproducing kernel on a certain Hilbert space.

\begin{tw}[Moore--Aronszajn]
Let $K$ be a positive definite kernel on a non-empty set $A$. Then there is a unique Hilbert space $H(K)$ of complex-valued functions on $A$ with the reproducing kernel $K$.
\label{MA}
\end{tw}
\begin{proof}(A sketch; for details, see \cite[Theorem 2.14]{paulsen})
One considers the vector space $H_{0}(K) := \textup{span} \{K_a\}_{a \in A}$. The map $\left\langle \cdot, \cdot\right\rangle_{H_{0}}: H_{0}(K) \times H_{0}(K) \rightarrow \mathbb{C}$ given by
\begin{align}
\label{iloczynskalarny}
\left\langle \sum^{m}_{l=1} \lambda_l K_{a_l}, \sum^{n}_{k=1} \beta_k K_{b_k} \right\rangle_{H_{0}} := \sum^{n}_{k=1} \sum^{m}_{l=1}\lambda_l \bar{\beta}_k K(b_k,a_l)
\end{align}
can be demonstrated to be a well-defined inner product. Then one shows that the completion of $H_0(K)$ in the norm induced by that inner product, denoted $H(K) := \widetilde{H_0(K)} = \widetilde{\textup{span}}\{K_a\}_{a \in A}$, can still be interpreted as a space of complex-valued functions on $A$, with $K$ as its reproducing kernel.
\end{proof}

In what follows, we shall be using the following general way of constructing positive definite kernels.
\begin{tw}
\label{tw2}
Let $F: A \rightarrow H$ be any function from a nonempty set $A$ to a Hilbert space $H$ equipped with an inner product $\langle \cdot, \cdot \rangle$. Then:
\begin{itemize}
\item The map $K: A \times A \rightarrow \mathbb{C}$ defined as $K(b,a) := \langle F(a), F(b) \rangle$ is a positive definite kernel.
\item By Theorem \ref{MA}, there exists an RKHS $H(K) := \widetilde{\textup{span}}\left\{ K(\cdot,a) \right\}_{a \in A}$, for which $K$ is the reproducing kernel.
\item The linear map $T: H \rightarrow H(K)$ defined via $(Tv)(a) := \langle v , F(a) \rangle$ is a surjective contraction. Moreover, it becomes an isometry when restricted to the closed subspace $V := \overline{\textup{span}}\, F(A)$ of the space $H$.
\item Defining $S: H(K) \rightarrow H$ as $S := T|_V^{-1}$, we obtain an isometric embedding of the RKHS $H(K)$ into $H$ that satisfies $S(K_a) = F(a)$ for every $a \in A$.
\end{itemize}
\end{tw}
\begin{proof}
Cf. \cite[p. 13]{alp}.
\end{proof}
Notice that the third bullet of the above theorem carries a lot of information about the functions belonging to $H(K)$. For example, if $F$ is weakly continuous, then $H(K) \subset C(A)$.

The (unique) reproducing kernel of a given RKHS turns out be tightly related to the so-called Parseval frames, which can be thought of as generalizations of complete orthonormal systems of vectors.
\begin{df}[cf. Definition 2.6 \& Proposition 2.8 in \cite{paulsen}]
Let $H$ be a Hilbert space (not necessarily of functions) with an inner product $\langle \cdot, \cdot \rangle$. The set $\{w_j\}_{j \in I} \subset H$ is called a \emph{Parseval frame} for $H$ if $\|v\|^2 = \sum_{j \in I} | \langle v, w_j \rangle |^2$ for every $v \in H$ or, equivalently, if $v = \sum_{j \in I} \langle v, w_j \rangle w_j$ for every $v \in H$.
\end{df}

\begin{tw}[Papadakis]
\label{tw3}
Let $H$ be an RKHS of functions on $A$ with reproducing kernel $K$. Then the family $\left\{ \varphi_j \right\}_{j \in I} \subset H$ is a Parseval frame for $H$ iff
\begin{align}
K(b,a)= \sum_{j \in I} \varphi_j(b) \overline{\varphi_j(a)}, \quad a,b \in A,
\label{uoz}
\end{align}
where the series converges pointwise.
\end{tw}
\begin{proof}
See \cite[Theorem 2.10]{paulsen}.
\end{proof}

\subsection{Reproducing kernel defined by the convolution of functions on a locally compact group}
\label{subsubsec:Splot funkcji na grupie lokalnie zwartej}

Let $G$ be a locally compact group (not necessarily abelian) and let $\mu$ be a fixed left Haar measure on $G$. For two continuous compactly supported complex-valued functions $f_{1}, f_{2} \in C_c(G)$, their convolution $f_{1} \ast f_{2}$ is another such function on $G$ defined via
\begin{align*}
(f_{1} \ast f_{2})(g) = \int_{G} f_{1}(\tau ) f_{2}(\tau^{-1} g) d \mu (\tau).
\end{align*}

The above definition extends to the space $L^2(G, \mu)$ in the following sense \cite[444O \& 444R]{fremlin}: for any $f_1,f_2 \in L^2(G, \mu)$ the convolution $f_1 \ast \tilde{f}_2$, where $\tilde{f}_2(g) := \overline{f_2(g^{-1})}$, is a well-defined element of $C_b(G)$ and, moreover, $|(f_1 \ast \tilde{f}_2)(g)| \leq \|f_1\|_{L^2} \|f_2\|_{L^2}$ for every $g \in G$.

\begin{ex}
\label{Ex1}
Let $G$ be a locally compact group and $\mu$ be a left Haar measure on $G$. Fix $f \in L^2(G, \mu)$ and consider the map $K: G \times G \rightarrow \mathbb{C}$ defined via
\begin{align*}
K(h, g):= (f \ast \tilde{f})(g^{-1} h), \quad g, h \in G.
\end{align*}
By the preceding discussion, $K$ is bounded and jointly continuous. Moreover, it is a positive definite kernel, because for any $n \in \mathbb{N}$, $\lambda_1,\ldots,\lambda_n \in \mathbb{C}$ and $g_1,\ldots,g_n \in G$ one has that
\begin{align*}
\sum_{i,j=1}^{n} \lambda_i K(g_i,g_j) \overline{\lambda}_j & = \sum_{i,j=1}^{n} \lambda_i \overline{\lambda}_j \int_{G} f(\tau) \tilde{f}( \tau^{-1} g_j^{-1} g_i)  d\mu (\tau) =\sum_{i,j=1}^{n} \lambda_i \overline{\lambda}_j \int_{G} f(\tau) \overline{f ( g_i^{-1} g_j \tau )}  d\mu (\tau)
\\
&=\sum_{i,j=1}^{n} \lambda_i \overline{\lambda}_j \int_{G} f(g_j^{-1} \tau) \overline{f ( g_i^{-1} \tau )}  d\mu (\tau) =\int_{G} \left| \sum_{j=1}^{n} \overline{\lambda}_j f (g_j^{-1} \tau) \right|^2 d\mu (\tau) \geq 0,
\end{align*}
where in the antepenultimate step we employed the left-invariance of $\mu$.

By Theorem \ref{MA}, $K$ is a reproducing kernel on a certain Hilbert space $H(K)$ equipped with the inner product $\left\langle \cdot,\cdot \right\rangle_{H(K)}$ with the reproducing family $\{K_g := K(\cdot, g)\}_{g \in G}$. In fact, $H(K)$ is a certain space of bounded continuous functions on $G$, which can be isometrically embedded into $L^2(G, \mu)$. To see why this is the case, let $L_g: G \rightarrow G$, $\xi \mapsto g\xi$ denote the \emph{left shift} by the element $g \in G$. Employing the pullback $L_g^\ast f := f \circ L_g$, one can write that
\begin{align}
\label{jadrosplotowe}
K(h,g) = \int_{G} f(\tau) \tilde{f}( \tau^{-1} g^{-1} h ) d\mu(\tau) = \int_{G} f(g^{-1}\tau) \overline{f (h^{-1}\tau) } d\mu(\tau) = \left\langle L_{g^{-1}}^\ast f, L_{h^{-1}}^\ast f \right\rangle_{L^2}
\end{align}
for any $g,h \in G$. Thus, the above construction is actually a special case of the one described in Theorem 2, with the map $F: G \rightarrow L^2(G,\mu)$ defined as $F(g) := L_{g^{-1}}^\ast f$ for every $g \in G$. Hence the elements of $H(K)$ are mappings of the form $g \mapsto \langle v, L_{g^{-1}}^\ast f \rangle$ with $v \in L^2(G,\mu)$, which are clearly bounded (by $\|v\|_{L^2}\|f\|_{L^2}$) and continuous, whereas the isometric embedding $S: H(K) \rightarrow L^2(G,\mu)$ satisfies $S(K_g) = L_{g^{-1}}^\ast f$.

As a side remark, observe that the above positive definite kernel $K$ can also be expressed as
\begin{align*}
K(h,g) &= \int_{G} f(g^{-1}\tau) \overline{f (h^{-1}\tau) } d\mu(\tau)= \int_{G} \tilde{f}(\tau^{-1}h) \overline{\tilde{f} (\tau^{-1}g) } d\mu(\tau)
\\
&= \int_{G} L_{\tau^{-1}}^\ast \tilde{f}(h) \overline{L_{\tau^{-1}}^\ast \tilde{f}(g)} d\mu(\tau), \quad g, h \in G,
\end{align*}
what is analogous to formula (\ref{uoz}), only here instead of a Parseval frame $\left\{ \varphi_j \right\}_{j \in I}$ and a counting measure we have the family $\{L_{\tau^{-1}}^\ast \tilde{f}\}_{\tau \in G}$ of the left-translations of the map $\tilde{f}$ and the left Haar measure $\mu$.
\end{ex}

\section{Groupoids}
\label{sec:Groupoids}

Groupoids provide both a generalization of groups \cite{bro} and of equivalence relations \cite{pat}.

\begin{df}
Let $\Gamma,X$ be nonempty sets. A \textit{groupoid} $\Gamma$ over a set $X$ is a septuple $(\Gamma,X, s,r, \varepsilon, \cdot, {}^{-1})$ with the below described mappings:
\begin{enumerate}
	\item the \emph{source} mapping $s : \Gamma \longrightarrow X$, which is surjective.
	\item the \emph{range} mapping $r : \Gamma \longrightarrow X$, which is surjective.
	\item the \emph{multiplication} mapping $\cdot : \Gamma^{(2)} \longrightarrow \Gamma$, where $\Gamma^{(2)}:=\left\{(\gamma_{1}, \gamma_{2}) \in \Gamma \times \Gamma \ | \ r(\gamma_{2})=s(\gamma_{1})\right\}$.
The multiplication is associative. For simplicity, instead of $\gamma_{1} \cdot \gamma_{2}$ we write $\gamma_{1} \gamma_{2}.$
\item the \emph{embedding} map $\varepsilon : X \longrightarrow \Gamma$ fulfilling
	$$\varepsilon(r(\gamma))\gamma= \gamma=\gamma \varepsilon(s(\gamma)).$$
\item the \emph{inversion} map $^{-1} : \Gamma \longrightarrow \Gamma$ such that
\begin{align*}
\forall_{\gamma \in \Gamma} \quad \gamma^{-1} \gamma=\varepsilon(s(\gamma)) \quad \textrm{and} \quad \gamma \gamma^{-1} =\varepsilon(r(\gamma)).\end{align*}

\end{enumerate}
\end{df}

The element $\gamma \in \Gamma$ can be regarded as an arrow with the starting point at $x=s(\gamma)$ and the ending point at $y=r(\gamma)$, where $x,y\in X$.
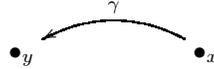
\begin{displaymath}
\xymatrix{
\bullet_{y}  \ar@/^1pc/@{<-}[rr]^\gamma &&
\bullet_{x}
}
\end{displaymath}

\begin{uw}
From the above definition it follows that the groupoid $(\Gamma,X, s,r, \varepsilon, \cdot, {}^{-1})$ can be identified with a small category in which all morphisms are invertible. In this interpretation $X$ is the set of objects, $\Gamma$ is the set of morphisms, $s(\gamma)$ is the domain of $\gamma$, $r(\gamma)$ is the codomain of $\gamma$, $\varepsilon(x)$ is the identity morphism at $x$, $\cdot$ is the composition of morphisms and $\gamma^{-1}$ is the inverse of the morphism $\gamma \in \Gamma.$
\end{uw}

\begin{ex}
A group $G$ is a groupoid over the singleton $X:=\left\{x\right\}$ with the mappings defined as $\forall_{g \in G} \quad s(g):=x, r(g):=x$ and $\varepsilon(x) := e$ --- the unit of the group $G$. The inversion and multiplication mappings are given by the respective group operations.
\end{ex}

Let $\Gamma$ be a groupoid over the set $X$. It is easy to show that the relation $\sim$ on the set $X$ defined as follows:
$$\forall_{x,y \in X} \quad x \sim y \quad \Leftrightarrow \quad \exists_{\gamma \in \Gamma} \quad x=s(\gamma) \ \ \wedge \ \ y=r(\gamma)$$
is an equivalence relation on $X$. The equivalence class of any element $x$ with respect to $\sim$ is the set $$\textrm{Or}_{x} := \left[x\right]_{\sim} = \left\{y \in X \ | \ \exists_{\gamma \in \Gamma} \ \ s(\gamma)=x \ \wedge \ r(\gamma)=y \right\}= r(\Gamma_{x}),$$ where $\Gamma_{x} := \left\{\gamma \in \Gamma \ | \ s(\gamma)=x\right\}=s^{-1}(x).$

We call $\textrm{Or}_{x}$ \textit{the orbit} of the element $x \in X$ with respect to the groupoid $\Gamma.$ We say that $\Gamma$ is transitive if for any $x,y \in X$ we have $x \sim y$ (or, equivalently, if $\textrm{Or}_{x}=X$ for some and hence for every $x \in X$).

Defining \textit{the fibers} of the groupoid $\Gamma$ as $\Gamma_{x} := s^{-1}(x)$ and $\Gamma^{x}:=r^{-1}(x)$ and using the above definition of the equivalence relation $\sim$, we have $r(s^{-1}(x))=  s(r^{-1}(x))$.

The set of arrows starting at $x$ and ending at $y$ is denoted by $\Gamma_{x}^{y}:=\Gamma_{x} \cap \Gamma^{y}.$
The set $\Gamma_{x}^{x}:=\Gamma_{x} \cap \Gamma^{x}=\left\{\gamma \in \Gamma \ | \ r(\gamma)= s(\gamma)=x \right\}$ of elements starting and ending at $x$ together with the groupoid multiplication and inversion map has a group structure, and we call it the \textit{isotropy group} of the element $x$.

\begin{uw}
If $\Gamma$ is a transitive groupoid over $X$ then it is easy to show that for any $y \in X$ the groups $\Gamma_{x}^{x}$ and $\Gamma_{y}^{y}$ are isomorphic (albeit not canonically!). Without the transitivity, however, it may happen that for some $x \neq y$ the fibers $\Gamma_{x}^{x}$ and $\Gamma_{y}^{y}$ are not isomorphic as groups.
\end{uw}

\begin{uw}
If $(\Gamma,X, s,r, \varepsilon, \cdot, {}^{-1})$  is a groupoid and $\Gamma, X$ are topological spaces then we call $\Gamma$ \textit{a topological groupoid} if all maps $s,r, \varepsilon, \cdot, ^{-1}$ are continuous. In this case the maps $\varepsilon, ^{-1}$ are homeomorphisms onto their images and all isotropy groups are topological groups.
\end{uw}

\subsection{Unitary representation of a groupoid}
\label{subsec:Reprezentacja unitarna grupoidu}

Recall that a \emph{unitary representation} of a group $G$ is a pair $(\mathcal{U},H)$, where $H$ is a Hilbert space and $\mathcal{U}$ is a map assigning to every $g \in G$ a unitary operator $\mathcal{U}(g): H \rightarrow H$ satisfying $\mathcal{U}(g_1g_2) = \mathcal{U}(g_1)\mathcal{U}(g_2)$ for all $g_1,g_2 \in G$. A standard example of a unitary representation of a (locally compact) group is the left regular representation $\mathcal{U}(g) := L^{\ast}_{g^{-1}}: L^2(G,\mu) \rightarrow L^2(G,\mu)$, which appeared in Example \ref{Ex1} above.

The notion of a unitary representation can be extended onto groupoids.
\begin{df}
\label{DefRepGrupoidu}
Let $\Gamma$ be a groupoid over the set $X$. A pair $(\mathcal{U},\mathbb{H})$ is called a \emph{unitary representation} of the groupoid $\Gamma$ if $\mathbb{H} := \{H_x\}_{x \in X}$ is a family of Hilbert spaces and $\mathcal{U}$ is a mapping assigning to each $\gamma \in \Gamma$ a unitary transformation $\mathcal{U}(\gamma): H_{s(\gamma)} \rightarrow H_{r(\gamma)}$ in such a way that $\mathcal{U}(\gamma_1 \gamma_2) = \mathcal{U}(\gamma_1)\mathcal{U}(\gamma_2)$ for every $(\gamma_1, \gamma_2) \in \Gamma^{(2)}$.
\end{df}
Notice that from the above definition it already follows that
\begin{itemize}
\item $\mathcal{U}(\varepsilon(x)) = \textup{id}_{H_x}$ for every $x \in X$.
\item $\mathcal{U}(\gamma^{-1}) = \mathcal{U}(\gamma)^{-1} = \mathcal{U}(\gamma)^\dag$ for every $\gamma \in \Gamma$.
\end{itemize}

From the family $\mathbb{H} := \{H_x\}_{x \in X}$ one can construct a new Hilbert space $\mathcal{H} := \widetilde{\bigoplus}_{x \in X} H_x$ defined as the completion of the direct sum $\bigoplus_{x \in X} H_x$ with respect to the norm generated by the inner product $\sum_{x \in X} \langle \cdot, \cdot \rangle_{x}$, where $\langle \cdot, \cdot \rangle_{x}$ denotes the inner product on $H_x$.

Many authors \cite{dix,pat,pys2} use a more general and somewhat more involved definition of a unitary representation of a groupoid, in which $\mathbb{H}$ is a \emph{Hilbert bundle} or a \emph{measurable field of complex Hilbert spaces}. This allows for constructing many more interesting Hilbert spaces from the fibers $H_x$ than just the completed direct sum $\mathcal{H}$ above. In the present paper, however, we keep the definition simple so that the constructions of the positive definite kernels and RKHS become more tractable.

\begin{ex}
Let $\Gamma$ be a groupoid over $X$ and $H$ be a fixed Hilbert space equipped with an inner product $\langle \cdot,\cdot \rangle$. Consider the family $\{H_x := \{x\} \times H\}_{x \in X}$ with $H_x$ endowed with the inner product $\langle (x,h_1), (x,h_2) \rangle_x := \langle h_1,h_2 \rangle$. For any $\gamma \in \Gamma_{x}^{y}$, $x,y \in X$ the transformation $\mathcal{U}(\gamma): H_{x} \rightarrow H_{y}$ defined by $\mathcal{U} (\gamma)(x,h):=(y,h)$ (change of the fiber base point) is of course unitary. Any such a representation of the groupoid $\Gamma$ is called a \textit{trivial representation}.
\end{ex}

A less trivial standard examples of groupoid representations arise for locally compact topological groupoids. Observe that for any such a groupoid $\Gamma$, the fibers $\Gamma^{x}:=r^{-1}(x)$ and $\Gamma_{x}:=s^{-1}(x)$ for any $x \in X$ are locally compact spaces, too, and the same concerns the base space $X$ itself. On such groupoids one introduces the following generalization of Haar measures \cite{pat,pys2}.

\begin{df}
Let $\Gamma$ be a locally compact topological groupoid. The family $\left\{ \lambda^{x} \right\}_{x \in X}$ of regular Borel measures on $\Gamma^x$ is called a \textit{left Haar system} for the groupoid $\Gamma$ if
\begin{enumerate}[1.]
	\item For every $f \in C_{c}(\Gamma)$ the function $f^0: X \rightarrow \mathbb{C}$ given by $f^0(x):=\int_{\Gamma^{x}} f d \lambda^{x}$ is continuous,
	\item For every $\gamma \in \Gamma$ and every $f \in C_{c}(\Gamma)$
	\begin{align*}
	\int_{\Gamma^{s(\gamma)}} f (\gamma \chi) d \lambda^{s(\gamma)}(\chi)= \int_{\Gamma^{r(\gamma)}} f (\chi) d \lambda^{r(\gamma)}(\chi).
	\end{align*}
\end{enumerate}
Similarly, the family $\left\{ \lambda_{x} \right\}_{x \in X}$ of regular Borel measures on $\Gamma_x$ is called a \textit{right Haar system} for the groupoid $\Gamma$ if
\begin{enumerate}[1.{$^\prime$}]
	\item For every $f \in C_{c}(\Gamma)$ the function $f_0: X \rightarrow \mathbb{C}$ given by $f_0(x):=\int_{\Gamma_{x}} f d \lambda_{x}$ is continuous,
	\item For every $\gamma \in \Gamma$ and every $f \in C_{c}(\Gamma)$
	\begin{align*}
	\int_{\Gamma_{r(\gamma)}} f ( \chi \gamma) d \lambda_{r(\gamma)}(\chi)= \int_{\Gamma_{s(\gamma)}} f (\chi) d \lambda_{s(\gamma)}(\chi).
	\end{align*}
\end{enumerate}
Observe that both $f^0$ and $f_0$ are automatically compactly supported. Conditions 1. and 1.$^\prime$ express the demand for the measures $\lambda^x$ and $\lambda_x$ to vary continuously over $X$, whereas conditions 2. and 2.$^\prime$ generalize the notions of, respectively, left and right invariance of Haar measures on locally compact groups.
\end{df}

\begin{ex}
For any $x \in X$ let $H_{x} = L^{2}(\Gamma^{x}, \lambda^{x})$, where $\left\{ \lambda^{x} \right\}_{x \in X} $ is a left Haar system on $\Gamma$. For any $\gamma \in \Gamma_{x}^{y}$ and for $f \in H_{x}$ define a unitary transformation $\mathcal{U} (\gamma): H_{x} \rightarrow H_{y}$ by
\begin{align*}
( \mathcal{U} (\gamma) f ) (\chi) = f( \gamma^{-1} \chi ), \quad \chi \in \Gamma^{y}.
\end{align*}
Such a representation $(\mathcal{U}, \{H_x\}_{x \in X})$ is called \textit{the left regular representation} of the groupoid $\Gamma$.
\end{ex}

\begin{ex}
For any $x \in X$ let $H_{x} = L^{2}(\Gamma_{x}, \lambda_{x})$, where $\left\{ \lambda_{x} \right\}_{x \in X} $ is a right Haar system on $ \Gamma$. For any $\gamma \in \Gamma_{x}^{y}$ and for $f \in H_{x}$ define a unitary transformation $\mathcal{U} (\gamma): H_{x} \rightarrow H_{y}$ by $$( \mathcal{U} (\gamma) f ) (\chi) = f( \chi \gamma ), \quad \chi \in \Gamma_{y}.$$
Such a representation $(\mathcal{U}, \{H_x\}_{x \in X})$ is called \textit{the right regular representation} of the groupoid $\Gamma$.
\end{ex}

For more examples of unitary representations of groupoids, the Reader is referred to \cite{ibo1,pys2}.

\section{Reproducing kernels and unitary representations}
\label{sec:Results}
\subsection{Reproducing kernel associated to a unitary representation of a group}
\label{subsec:Results Groups}

Let $G$ be a group and $(\mathcal{U}, H)$ be its unitary representation on a Hilbert space $H$ equipped with an inner product $\langle \cdot, \cdot \rangle$. Additionally, let $v \in H$ be any fixed vector. Formula (\ref{jadrosplotowe}) in Example \ref{Ex1} suggests considering the kernel $K: G \times G \rightarrow \mathbb{C}$ defined as
\begin{align}
\label{jadro}
K(h,g) := \langle \mathcal{U}(g)v, \mathcal{U}(h)v \rangle.
\end{align}
Notice that thus defined $K$ is a special case of the construction presented in Theorem \ref{tw2} with the map $F: G \rightarrow H$ given by  $F(g) := \mathcal{U}(g)v$, $g \in G$. This means, in particular, that the RKHS $H(K)$ provided by the Moore--Aronszajn theorem is a space of functions of the form $g \mapsto \langle w, \mathcal{U}(g)v \rangle$, where $w \in H$. Notice that, if $\mathcal{U}$ is a weakly continuous unitary representation of a topological group $G$, then $H(K) \subset C(G)$. Moreover, if $\{w_j\}_{j \in I} \subset H$ is a Parseval frame for $H$, then
\begin{align*}
K(h,g) = \langle \mathcal{U}(g)v, \mathcal{U}(h)v \rangle = \sum_{j \in I} \langle \mathcal{U}(g)v, w_j \rangle \langle w_j, \mathcal{U}(h)v \rangle = \sum_{j \in I} \varphi_j(h) \overline{\varphi_j(g)},
\end{align*}
where $\varphi_j := Tw_j \in H(K)$ (cf. the third bullet in Theorem \ref{tw2}).  On the strength of Theorem \ref{tw3}, $\{\varphi_j\}_{j \in I}$ is a Parseval frame for $H(K)$.

By the unitarity of the representation, the kernel $K$ satisfies
\begin{align}
\label{wlasnosc}
K(h,g) = K(g^{-1}h,e) \textup{ for any } g,h \in G,
\end{align}
where $e$ is the unit element of $G$.

Conversely, suppose we are given a positive definite kernel $K: G \times G \rightarrow \mathbb{C}$ on a group $G$ satisfying (\ref{wlasnosc}). We can employ the Moore--Aronszajn theorem to define its representation $(\mathcal{U},H(K))$. Concretely, for any $g \in G$ define $\mathcal{U}(g)\sum_{i=1}^n \lambda_i K_{h_i} := \sum_{i=1}^n \lambda_i K_{gh_i}$ for any element of $H_0(K) := \textup{span} \{K_h\}_{h \in G}$. Since there is no guarantee that the system $\{K_h\}_{h \in G}$ is linearly independent, we must check that such a $\mathcal{U}(g)$ is well defined. To this end, it suffices to prove that if $\sum_{i=1}^n \lambda_i K_{h_i} \equiv 0$ on $G$, then also $\sum_{i=1}^n \lambda_i K_{gh_i} \equiv 0$. But thanks to (\ref{wlasnosc}) we have that, for any $h \in G$,
\begin{align*}
\sum_{i=1}^n \lambda_i K_{gh_i}(h) = \sum_{i=1}^n \lambda_i K_{h_i}(g^{-1}h) = 0
\end{align*}
by assumption. Since $H_0(K)$ is dense in $H(K)$, thus defined $\mathcal{U}(g)$ can be uniquely extended to $H(K)$. Moreover, also by (\ref{wlasnosc}) and by the density argument, one obtains the unitarity of $\mathcal{U}(g)$ by verifying that for any $h,h' \in G$
\begin{align*}
\langle \mathcal{U}(g)K_h, \mathcal{U}(g)K_{h'} \rangle_{H(K)} & = \langle K_{gh}, K_{gh'} \rangle_{H(K)} = K(gh',gh) = K((gh)^{-1}gh',e)
\\
& = K(h^{-1}g^{-1}gh',e) = K(h^{-1}h',e) = K(h,h') = \langle K_{h'}, K_{h} \rangle_{H(K)}.
\end{align*}

Observe, finally, that one can retrieve the kernel $K$ from the above `Moore--Aronszajn representation' through formula (\ref{jadro}). Indeed, one simply has to take $v := K_e$.

\subsection{Reproducing kernel associated to a unitary representation of a groupoid}
\label{subsec:Results Groupoids}

Let us now generalize the above relationship between unitary representations of groups and reproducing kernels onto the groupoid setting. Let thus $\Gamma$ be a groupoid over $X$ and $(\mathcal{U}, \mathbb{H} = \{H_x\}_{x \in X})$ be its unitary representation as specified by Definition \ref{DefRepGrupoidu}. Additionally, let $v$ be a fixed \emph{vector field}, by which we shall understand a mapping $X \ni x \mapsto v(x) \in H_x$ (note: vector fields need not belong to $\mathcal{H} := \widetilde{\bigoplus}_{x \in X} H_x$). Define the kernel $K: \Gamma \times \Gamma \rightarrow \mathbb{C}$ via
\begin{align}
\label{jadro2}
K(\chi, \gamma) := \left\{ \begin{array}{cc} \left\langle \mathcal{U}(\gamma)v(s(\gamma)), \mathcal{U}(\chi)v(s(\chi)) \right\rangle_{r(\gamma)} & \textup{for } r(\gamma) = r(\chi),
\\
0 & \textup{for } r(\gamma) \neq r(\chi).
\end{array} \right.
\end{align}
Observe that for each $x \in X$ the restriction $K^x := K|_{\Gamma^x \times \Gamma^x}$ is a positive definite kernel --- it constitutes a special case of the construction presented in Theorem \ref{tw2} with $F^x: \Gamma^x \rightarrow H_x$ given by $F^x(\gamma) := \mathcal{U}(\gamma)v(s(\gamma))$ for every $\gamma \in \Gamma^x$. Invoking the Moore--Aronszajn theorem (Theorem \ref{MA}), we obtain an RKHS given by $H(K^x) := \widetilde{\textup{span}}\left\{ K^x_\gamma := K^x(\cdot,\gamma) \right\}_{\gamma \in \Gamma^x}$, whose reproducing kernel is $K^x$. By the third bullet of Theorem \ref{tw2}, every element of $H(K^x)$ is a function on $\Gamma^x$ of the form $\gamma \mapsto \langle w_x, \mathcal{U}(\gamma)v(s(\gamma)) \rangle$, where $w_x \in H_x$. Finally, by the fourth bullet of Theorem \ref{tw2}, for any $x \in X$ the Hilbert space $H(K^x)$ can be isometrically embedded into $H_x$, its image being the closed subspace $V^x := \overline{\textup{span}} \{\mathcal{U}(\gamma)v(s(\gamma))\}_{\gamma \in \Gamma^x}$.
\\

Also the kernel $K$ itself is a realization of the general construction described in Theorem \ref{tw2}. To see this, consider $F: \Gamma \rightarrow \mathcal{H}$ given by $F(\gamma) := \mathcal{U}(\gamma)v(s(\gamma))$ and observe that $\langle F(\gamma), F(\chi) \rangle_{\mathcal{H}} := \sum_{x \in X} \langle F(\gamma), F(\chi) \rangle_x$ indeed equals $K(\chi,\gamma)$ for all $\gamma, \chi \in \Gamma$, because $H_x \bot H_y$ as subspaces of $\mathcal{H}$ for $x \neq y$. Therefore, $K$ is positive definite and the Moore--Aronszajn theorem yields an RKHS defined as $H(K) := \widetilde{\textup{span}}\left\{ K_\gamma := K(\cdot,\gamma) \right\}_{\gamma \in \Gamma}$.

Notice that the two above constructions are of RKHS compatible in the sense that
\begin{align}
\label{MA2odslony}
H(K) = \widetilde{\bigoplus}_{x \in X} H(K^x).
\end{align}
Indeed, observe that both spaces have the same dense subspaces, namely (cf. the sketch of the proof of Theorem \ref{MA} above)
\begin{align*}
H_0(K) := \textup{span}\left\{ K_\gamma \right\}_{\gamma \in \Gamma} = \bigoplus_{x \in X} \textup{span}\left\{ K_\gamma \right\}_{\gamma \in \Gamma^x} = \bigoplus_{x \in X} \textup{span}\left\{ K^x_\gamma \right\}_{\gamma \in \Gamma^x} =: \bigoplus_{x \in X} H_0(K^x),
\end{align*}
where the maps $K_\gamma: \Gamma \rightarrow \mathbb{C}$ and $K^x_\gamma: \Gamma^x \rightarrow \mathbb{C}$ have been identified (the former being the extension by zero of the latter). Since both above spaces are equipped with the same inner product (given by (\ref{iloczynskalarny})), they yield the same Hilbert spaces after completion.

Similarly as before, by the third bullet of Theorem \ref{tw2}, every element of $H(K)$ is a function on $\Gamma$ of the form $\gamma \mapsto \langle w, \mathcal{U}(\gamma)v(s(\gamma)) \rangle$, where $w \in \mathcal{H}$. Notice that such functions must vanish on all but countably many fibers $\Gamma^x$. Moreover, if $\{w_j\}_{j \in I} \subset \mathcal{H}$ is a Parseval frame for $\mathcal{H}$, then
\begin{align*}
K(\chi,\gamma) & = \langle \mathcal{U}(\gamma)v(s(\gamma)), \mathcal{U}(\chi)v(s(\chi)) \rangle_{\mathcal{H}} = \sum_{j \in I} \langle \mathcal{U}(\gamma)v(s(\gamma)), w_j \rangle_{\mathcal{H}} \langle w_j, \mathcal{U}(\chi)v(s(\chi)) \rangle_{\mathcal{H}}
\\
& = \sum_{j \in I} \varphi_j(\chi) \overline{\varphi_j(\gamma)},
\end{align*}
where $\varphi_j := \langle w_j, \mathcal{U}(\cdot)v(s(\cdot)) \rangle_{\mathcal{H}}$. Similarly as in the group case, by Theorem \ref{tw3} we obtain that $\{\varphi_j\}_{j \in I}$ is a Parseval frame for $H(K)$.

Finally, by the fourth bullet of Theorem \ref{tw2}, the Hilbert space $H(K)$ can be isometrically embedded into $\mathcal{H}$, its image being the closed subspace $V := \overline{\textup{span}} \{\mathcal{U}(\gamma)v(s(\gamma))\}_{\gamma \in \Gamma} = \overline{\bigoplus}_{x \in X} V^x$, where the last equality can be proven completely analogously to (\ref{MA2odslony}).

Additionally, the unitarity of the representation implies that for $\gamma, \chi \in \Gamma$ such that $r(\gamma) = r(\chi)$
\begin{align}
\label{wlasnosc2}
K(\chi,\gamma) = K(\gamma^{-1}\chi, \varepsilon(r(\gamma^{-1}\chi))),
\end{align}
which is nothing but a straightforward generalization of (\ref{wlasnosc}). Indeed, one has
\begin{align*}
K(\chi,\gamma) & = \left\langle \mathcal{U}(\gamma)v(s(\gamma)), \mathcal{U}(\chi)v(s(\chi)) \right\rangle_{r(\gamma)} = \left\langle v(s(\gamma)), \mathcal{U}(\gamma^{-1}\chi) v(s(\chi)) \right\rangle_{s(\gamma)}
\\
& = K(\gamma^{-1}\chi, \varepsilon(s(\gamma))) = K(\gamma^{-1}\chi, \varepsilon(r(\gamma^{-1}\chi))).
\end{align*}

Let us now consider the converse problem. That is, given a groupoid $\Gamma$ over $X$ and a positive definite kernel $K: \Gamma \times \Gamma \rightarrow \mathbb{C}$ such that $K(\gamma, \chi) = 0$ if $r(\gamma) \neq r(\chi)$ and (\ref{wlasnosc2}) holds, we construct the `Moore--Aronszajn representation' $(\mathcal{U}, \{H(K^x)\}_{x \in X})$ of $\Gamma$. To this end, define each $\mathcal{U}(\gamma): H(K^{s(\gamma)}) \rightarrow H(K^{r(\gamma)})$ first on the dense subspace $H_0(K^{s(\gamma)}) := \textup{span}\{K_\chi\}_{\chi \in \Gamma^{s(\gamma)}}$ by
\begin{align*}
\mathcal{U}(\gamma)\sum_{i = 1}^n \lambda_i K_{\chi_i} := \sum_{i = 1}^n \lambda_i K_{\gamma \chi_i},
\end{align*}
where, similarly as for groups, one can easily check that this definition is sound (analogously as in the group case, it is here where property (\ref{wlasnosc2}) steps in). Observe that $r(\chi_i) = s(\gamma)$ for all $i=1,\ldots,n$, so the products $\gamma \chi_i$ are all well defined. It is also straightforward to prove (again, thanks to (\ref{wlasnosc2})) that thus defined $\mathcal{U}(\gamma)$ preserves inner products. Extending it onto the entire $H(K^{s(\gamma)})$, by the arbitrariness of $\gamma$ we obtain the desired unitary representation of $\Gamma$.

Finally, notice that the kernel $K$ can be retrieved from the `Moore--Aronszajn representation' through formula (\ref{jadro2}), where one has to take the vector field $v(x) := K_{\varepsilon(x)}$, $x \in X$.

Before moving to examples, let us remark that formula (\ref{jadro2}) `works well' with the basic algebraical operations on the groupoid representations.
\begin{uw}
Let $\Gamma$ be a groupoid over $X$ and let $(\mathcal{U}^{(j)}, \mathbb{H}^{(j)} := \{H^{(j)}_{x}\}_{x \in X})$, $j=1,2,\ldots,l$, be its unitary representations. Fixing $l$ vector fields $v_1,v_2,\ldots,v_l$ on $X$ such that $v_j(x)\in  H^{(j)}_{x}$ for $x\in X$ and $j=1,2,\dots,l$ we can construct positive definite kernels $K_1,K_2,\dots,K_l$ on $\Gamma$, respectively, using formula (\ref{jadro2}). Consider now the direct sum $(\bigoplus_{j=1}^l\mathcal{U}^{(j)}, \bigoplus_{j=1}^l\mathbb{H}^{(j)} := \{\bigoplus_{j=1}^l H^{(j)}_{x}\}_{x \in X})$ of the above representations, and, taking the vector field $v^\oplus := \bigoplus_{j=1}^l v_j$, define a kernel $K^\oplus$ on $\Gamma$ via (\ref{jadro2}). Similarly, consider the tensor product $(\bigotimes_{j=1}^l\mathcal{U}^{(j)}, \bigotimes_{j=1}^l\mathbb{H}^{(j)} = \{\bigotimes_{j=1}^l H^{(j)}_{x}\}_{x \in X})$ of the representations and, taking this time the vector field $v^\otimes := \bigotimes_{j=1}^l v_j$, define another kernel $K^\otimes$ on $\Gamma$, again employing formula (\ref{jadro2}). It follows from \cite[p. 16,17]{alp} that
\begin{align*}
K^{\oplus}=K_1 + K_2 + \ldots + K_l \qquad \textup{and} \qquad K^\otimes=K_1 \cdot K_2 \cdot \ldots \cdot K_l.
\end{align*}
\end{uw}

\begin{uw}
Suppose now we have a family of groupoids $\{(\Gamma_i,X_i, s_i,r_i, \varepsilon_i, \cdot_i, {}^{-1_i})\}$ indexed by $i \in I$. Let also $(\mathcal{U}^{(i)}, \mathbb{H}^{(i)} := \{H^{(i)}_{x}\}_{x \in X_i})$ be a unitary representation of the groupoid $\Gamma_i$ for every $i \in I$. The disjoint union $\Gamma:=\bigsqcup_{i\in I}\Gamma_i$ has a natural groupoid structure with $X:=\bigsqcup_{i\in I}X_i$ as a base space, $\Gamma^{(2)}:=\bigsqcup_{i\in I}\Gamma_i^{(2)}$, $\gamma_1\cdot\gamma_2:=\gamma_1\cdot_i\gamma_2$ for $\gamma_1,\gamma_2\in\Gamma_i^{(2)}$, $\varepsilon(x):=\varepsilon_i(x)$ for $x\in X_i$ and, moreover, $s(\gamma):=s_i(\gamma)$, $r(\gamma):=r_i(\gamma)$, $\gamma^{-1}:=\gamma^{{-1}_i}$ for $\gamma\in\Gamma_i$. 

Consider now the family of Hilbert spaces $\mathbb{H} := \{H_x\}_{x \in X}$, where we put $H_x:=H^{(i)}_x$ for $x \in X_i$ and define a mapping $\mathcal{U}$ assigning to each $\gamma \in \Gamma$ a unitary transformation $\mathcal{U}(\gamma): H_{s(\gamma)} \rightarrow H_{r(\gamma)}$ by $\mathcal{U}(\gamma):=\mathcal{U}^{(i)}(\gamma)$ for $\gamma\in\Gamma_i$. Then $(\mathcal{U},\mathbb{H})$ is a unitary representation of $\Gamma$.

For any $i\in I$ let us fix a vector field $X_i \ni x \to v_i(x) \in H^{(i)}_x$ and use formula (\ref{jadro2}) to obtain a positive definite kernel $K_i$ from the unitary representation $(\mathcal{U}^{(i)}, \mathbb{H}^{(i)})$ of $\Gamma_i$. Additionally, introduce a vector field $v$ on $X$ by setting $v(x):=v_i(x)$ if $x\in X_i$ and use it together with the representation $(\mathcal{U},\mathbb{H})$ to obtain, again via (\ref{jadro2}), another positive definite kernel $K$ on $\Gamma$. It is not difficult to observe that
\begin{align}
K(\chi,\gamma) = \left\{ \begin{array}{cl} K_i(\chi,\gamma) & \textup{for } \chi,\gamma \in \Gamma_i,
\\
0 & \textup{for } \chi \in \Gamma_i, \gamma \in \Gamma_{i'} \textup{ with } i \neq i'.
\end{array} \right.
\end{align}
What is more, reasoning analogously as when proving formula (\ref{MA2odslony}), one can show that the RKHS $H(K)$ obtained from $K$ by means of the Moore--Aronszajn theorem satisfies $H(K)=\widetilde{\bigoplus}_{i \in I} H(K_i)$.  
\end{uw}

\begin{ex}
\label{ex:trivial}
Fix a Hilbert space $H$ endowed with the inner product $\langle \cdot, \cdot \rangle$ and let $(\mathcal{U}, \{H_x := \{x\} \times H\}_{x \in X})$ be the trivial representation of the groupoid $\Gamma$. For any fixed vector field $v$, which here can be regarded as an element of $H^X$, we have $\mathcal{U}(\gamma)v(s(\gamma))=v(r(\gamma))$ and formula (\ref{jadro2}) yields the kernel
\begin{align*}
K(\chi, \gamma) = \left\{ \begin{array}{cc}  \left\langle v(r(\gamma)), v(r(\chi)) \right\rangle =  \left\| v(r(\gamma)) \right\|^{2}& \textup{for } r(\gamma) = r(\chi),
\\
0 & \textup{for } r(\gamma) \neq r(\chi).
\end{array} \right.
\end{align*}
In other words, for every $x \in X$ the kernel's restriction $K^x$ is a constant map. Every $H(K^x)$ is thus either one-dimensional (if $v(x) \neq 0$) or zero-dimensional (if $v(x) = 0$). Moreover, if $v$ is a nowhere-vanishing vector field, the set $\{\|v(x)\| \mathbf{1}_{\Gamma^x} \}_{x \in X}$ (where $\mathbf{1}_{\Gamma^x}: \Gamma \rightarrow \{0,1\}$ denotes the indicator function of the fiber $\Gamma^x$) constitutes an orthonormal basis of $H(K)$, and hence the latter Hilbert space is isometrically isomorphic to $l^2(X)$.
\end{ex}

\begin{ex}
\label{qubitex}
As a simple nontrivial example, consider the four-element groupoid $\Gamma := \{\varepsilon(+), \varepsilon(-),$ $\alpha, \alpha^{-1} \}$ over a two-element set $X := \{+,-\}$, visualized in Figure \ref{qubitfig}. Such a groupoid (the pair groupoid over a two-element set) is studied e.g. in the context of quantum information \cite{ibo}.

\begin{figure}[h!]
\begin{displaymath}
\xymatrix{
\bullet_{-} \ar@(ul,dl)[]_{\varepsilon(-)} \ar@/^1pc/@{<-}[rr]^{\alpha}
\ar@/_1pc/@{->}[rr]_{\alpha^{-1}}
&& \bullet_{+} \ar@(ur,dr)[]^{\varepsilon(+)}
}
\end{displaymath}
\caption{Structure of the groupoid $\Gamma$ in Example \ref{qubitex}.}
\label{qubitfig}
\end{figure}
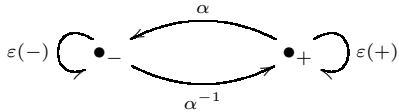

To further simplify the example, let us consider a one-dimensional representation of $\Gamma$. Concretely, let us take $H_+ := \{+\} \times \mathbb{C}$, $H_- := \{-\} \times \mathbb{C}$ and let the unitary representation $(\mathcal{U},\{H_+,H_-\})$ be given by
\begin{align*}
\mathcal{U}(\alpha): H_+ \rightarrow H_-, \quad \mathcal{U}(\alpha)(+,z) := (-,\lambda z)
\end{align*}
and hence $\mathcal{U}(\alpha^{-1}): H_- \rightarrow H_+, \ \mathcal{U}(\alpha^{-1})(-,z) := (+,\bar{\lambda} z)$, where $\lambda$ is a fixed complex number of modulus $1$. Choosing a generic vector field $v(+) := (+,v_+)$, $v(-) := (-,v_-)$, where $v_\pm \in \mathbb{C}$, formula (\ref{jadro2}) yields a kernel whose respective values are presented in Table \ref{tab1}.

\setlength\dashlinedash{0.6pt}
\setlength\dashlinegap{1.5pt}
\setlength\arrayrulewidth{0.3pt}

\begin{table}[h!]
\begin{center}
\begin{tabular}{|c|c c : c c|} \hline
 $K(\cdot,\cdot)$& $\varepsilon(+) $ &	$\alpha^{-1}$ &	$\alpha $	& $\varepsilon(-) $ \\ \hline
 $\varepsilon(+)$ & $ \left|v_{+}\right|^{2} $ &	$\bar{\lambda} v_{-} \bar{v}_{+}$ & 0 & 0 \\
 $\alpha^{-1}$ & $ \lambda \bar{v}_{-}v_{+}$ & $\left| v_{-}\right|^{2} $ & 0 &  0 \\ \hdashline
 $\alpha$ &  0 & 0 & $ \left|v_{+}\right|^{2}$ & $\bar{\lambda} v_{-}\bar{v}_{+}$ \\
 $\varepsilon(-)$ &  0 & 0  & $\lambda \bar{v}_{-}v_{+}$ & $\left| v_{-} \right|^{2}$ \\ \hline
\end{tabular}
\caption{The kernel obtained from the representation studied in Example \ref{qubitex}.}\label{tab1}
\end{center}
\end{table}
Notice that the columns of the above table contain the values of the functions $K_{\varepsilon(+)}$, $K_{\alpha^{-1}}$, $ K_\alpha$, $K_{\varepsilon(-)}$, respectively. These four functions span the space $H(K)$, but they are not linearly independent. In fact, one can write that $H(K) = \textup{span}\{\varphi_+, \varphi_-\}$, where the functions $\varphi_\pm: \Gamma \rightarrow \mathbb{C}$ are defined as
\begin{align*}
& \varphi_+(\varepsilon(+)) := \bar{v}_+, && \varphi_+(\alpha^{-1}) := \lambda \bar{v}_-, && \varphi_+(\alpha) := 0, && \varphi_+(\varepsilon(-)) := 0,
\\
& \varphi_-(\varepsilon(+)) := 0, && \varphi_-(\alpha^{-1}) := 0, && \varphi_-(\alpha) := \bar{\lambda} \bar{v}_+, && \varphi_-(\varepsilon(-)) := \bar{v}_-.
\end{align*}
Unless $v_+ = v_- = 0$, the functions $\varphi_\pm$ can be shown to be orthonormal:
\begin{align*}
\langle \varphi_+, \varphi_+ \rangle_{H(K)} = \langle \varphi_-, \varphi_- \rangle_{H(K)} = 1 \quad \textup{and} \quad \langle \varphi_+, \varphi_- \rangle_{H(K)} = \langle \varphi_-, \varphi_+ \rangle_{H(K)} = 0.
\end{align*}
Therefore, we have $H(K) \cong \mathbb{C}^2$ as Hilbert spaces, and so in this case $H(K)$ is \emph{isomorphic} to $\mathcal{H} := H_+ \oplus H_-$ and not just isometrically embedded in the latter (cf. Theorem \ref{tw2}). In addition, observe that the (restricted) functions $\varphi_+|_{\Gamma^+},\varphi_-|_{\Gamma^-}$ span the RKHS's $H(K^+)$, $H(K^-)$, respectively, built from the restricted kernels. All in all, we have that
\begin{align*}
H(K) = \textup{span}\{\varphi_+, \varphi_-\} = \textup{span}\{\varphi_+\} \oplus \textup{span}\{\varphi_-\} = H(K^+) \oplus H(K^-),
\end{align*}
where again we have identified the restricted functions with its extensions by zero, in full agreement with formula (\ref{MA2odslony}).

Finally, the set $\{\varphi_+,\varphi_-\}$, being an orthonormal basis of $H(K)$, is a Parseval frame for $H(K)$, and hence by Theorem \ref{tw3} we must have, for every $\chi,\gamma \in \Gamma$,
\begin{align*}
K(\chi,\gamma) = \varphi_+(\chi)\overline{\varphi_+(\gamma)} + \varphi_-(\chi)\overline{\varphi_-(\gamma)},
\end{align*}
As one can check directly, the above equality indeed holds in the considered case.
\end{ex}

\begin{ex}
Let $\Gamma$ be a locally compact groupoid endowed with a left Haar system $\{\lambda^x\}_{x \in X}$, and let $(\mathcal{U}, \{H_x := L^2(\Gamma^x, \lambda^x)\}_{x \in X})$ be its left regular representation, i.e. $(\mathcal{U}(\gamma)f)(\xi) := f(\gamma^{-1}\xi)$ for any $f \in H_{r(\gamma)}$.

For any fixed vector field $X \ni x \mapsto v(x) \in H_x$ the reproducing kernel reads, in the case when $r(\chi) = r(\gamma)$,
\begin{align*}
K(\chi, \gamma) & = \int_{\Gamma^{r(\gamma)}} v(s(\gamma))(\gamma^{-1}\xi) \overline{ v(s(\chi))(\chi^{-1}\xi)} d \lambda^{r(\gamma)} (\xi)
\\
& = \int_{\Gamma^{s(\gamma)}} v(s(\gamma))(\xi) \overline{ v(s(\chi))(\chi^{-1} \gamma \xi)} d \lambda^{s(\gamma)} (\xi)
\\
& = \int_{\Gamma^{s(\gamma)}} v(s(\gamma))(\xi) \widetilde{ v(s(\chi))}(\xi^{-1}\gamma^{-1} \chi) d \lambda^{s(\gamma)} (\xi).
\end{align*}

We note that this is an analogue of the kernel defined by the convolution on a group with respect to a left Haar measure (cf. Example \ref{Ex1}). Although there exists a standard definition of a convolution on a groupoid (see, e.g., \cite[p. 38]{pat}), the above expression does not entirely fit into it, because the convoluted functions $v(s(\gamma))$ and $\widetilde{v(s(\chi))}$ are \emph{not} defined over entire $\Gamma$. In fact, $v(s(\gamma)) \in L^{2}(\Gamma^{s(\gamma)}, \lambda^{s(\gamma)}) $, whereas $\widetilde{ v(s(\chi))} \in L^{2}(\Gamma_{s(\chi)}, \textup{inv}_{\ast} \lambda^{s(\chi)})$ (where $\textup{inv}$ denotes here the inversion map, $\textup{inv}(\gamma) := \gamma^{-1}$) and as such their convolution is well defined only on $\Gamma^{s(\gamma)}_{s(\chi)}$.

Of course, when $r(\gamma) \neq r(\chi)$ the kernel is by definition $K(\chi, \gamma) =0$.
\end{ex}

\begin{ex}
Let $\Gamma$ be a locally compact groupoid endowed this time with a right Haar system $\{\lambda_x\}_{x \in X}$, and let $(\mathcal{U}, \{H_x := L^2(\Gamma_x, \lambda_x)\}_{x \in X})$ be its right regular representation, i.e. $(\mathcal{U}(\gamma)f)(\xi) := f(\xi\gamma)$ for any $f \in H_{r(\gamma)}$.

For any fixed vector field $X \ni x \mapsto v(x) \in H_x$ the reproducing kernel reads, in the case when $r(\chi) = r(\gamma)$,
\begin{align*}
  K(\chi, \gamma) & =\int_{\Gamma_{r(\gamma)}} v(s(\gamma))(\xi \gamma) \overline{ v(s(\chi))(\xi \chi)} d \lambda_{r(\gamma)} (\xi)
\\
& =\int_{\Gamma_{s(\gamma)}} v(s(\gamma))(\xi) \overline{ v(s(\chi))(\xi \gamma^{-1} \chi )} d \lambda_{s(\gamma)} (\xi)
\\
& = \int_{\Gamma_{s(\gamma)}} v(s(\gamma))(\xi) \widetilde{ v(s(\chi))}(\chi^{-1}\gamma \xi^{-1}) d \lambda_{s(\gamma)} (\xi).
\end{align*}
This also can be seen as something analogous to the convolution (only this time, related to the right Haar measure in the group setting).

Of course, when $r(\gamma) \neq r(\chi)$ the kernel $K(\chi, \gamma)$ is defined to vanish, just as in the previous examples.
\end{ex}

\begin{ex}
This example is inspired by \cite{BP-W}. Let $X$ be the set of all pairs $(\Omega,z)$, where $\Omega$ is a domain (nonempty, open and connected set) in ${\mathbb C}^n$ and $z\in\Omega$. Let also $\Gamma$ be the set of all pairs $(\Phi,z)$, where $\Phi: \Omega_1\to\Omega_2$ is a biholomorphism between open domains in ${\mathbb C}^n$ and $z\in\Omega_1$. We define the groupoid $(\Gamma,X,s,r,\varepsilon, \cdot, {}^{-1})$ as follows.

If $(\Phi,z)\in\Gamma$, where $\Phi: \Omega_1\to\Omega_2$, then $s(\Phi,z):=(\Omega_1,z)$ and $r(\Phi,z):=(\Omega_2,\Phi(z))$. The embedding map is given by $\varepsilon(\Omega,z):=(\textup{id}_{\Omega},z)$, whereas the inversion map reads $(\Phi,z)^{-1}:=(\Phi^{-1},\Phi(z))$. Finally, the multiplication $\cdot$ of pairs $(\Phi: \Omega_1\to\Omega_2,z)$ and $(\Psi: \Omega_3\to\Omega_4,\zeta)$ is defined if $\Omega_2=\Omega_3$ and $\zeta=\Phi(z)$, in which case $(\Psi,\zeta)(\Phi,z):=(\Psi \circ \Phi,z)$, where $\circ$ is an ordinary composition of mappings.

In order to introduce a kernel on $\Gamma$, define a map $k:\Gamma \to \mathbb{C}$ via
\begin{align*}
k(\Phi,z):=\frac{J\Phi(z)}{|J\Phi(z)|},\quad (\Phi,z)\in\Gamma,
\end{align*}
where $J\Phi(z)$ denotes the complex Jacobian determinant of $\Phi$ at the point $z$. Notice that the map $k$ is multiplicative, i.e., $k(\gamma \chi) = k(\gamma) k(\chi)$ for any $(\gamma, \chi) \in \Gamma^{(2)}$. Indeed, using elementary properties of the Jacobian, we can write
\begin{align}
\label{kmultiplicative}
k((\Psi,\zeta)(\Phi,z)) = k(\Psi \circ \Phi,z) = \frac{J(\Psi \circ \Phi)(z)}{|J(\Psi \circ \Phi)(z)|} = \frac{J(\Psi)(\zeta)}{|J(\Psi)(\zeta)|} \cdot \frac{J(\Phi)(z)}{|J(\Phi)(z)|} = k(\Psi,\zeta)k(\Phi,z),
\end{align}
where $\zeta = \Phi(z)$ by the multiplicability of $(\Psi,\zeta)$, $(\Phi,z)$. What is more, for any $(\Omega,z) \in X$ one trivially has that $k(\varepsilon(\Omega,z)) = k(\textup{id}_\Omega,z) = 1$ and hence, moreover,
\begin{align}
\label{kmultiplicative2}
k(\gamma^{-1}) = k(\gamma)^{-1} k(\gamma) k(\gamma^{-1}) = k(\gamma)^{-1} k(\gamma \gamma^{-1}) = k(\gamma)^{-1} k(\varepsilon(r(\gamma))) = k(\gamma)^{-1} = \overline{k(\gamma)}
\end{align}
for every $\gamma \in \Gamma$.
 
Let now $K:\Gamma\times\Gamma\to{\mathbb C}$ be a kernel given by the formula
\begin{align}
\label{jadroJF1}
K(\chi,\gamma) := \left\{ \begin{array}{cc} k(\gamma)\overline{k(\chi)} & \textup{for } r(\gamma) = r(\chi),
\\
0 & \textup{for } r(\gamma) \neq r(\chi),
\end{array} \right.
\end{align}
that is
\begin{align*}
K((\Psi,\zeta),(\Phi,z)) := \left\{ \begin{array}{cc} \frac{J\Phi(z)}{|J\Phi(z)|}\frac{\overline{J\Psi(\zeta)}}{|J\Psi(\zeta)|} & \textup{for } r(\Phi,z) = r(\Psi,\zeta),
\\
0 & \textup{for } r(\Phi,z)\neq r(\Psi,\zeta).
\end{array} \right.
\end{align*}
One can easily show that the above kernel is positive definite. In fact, it is a realization of the general construction described in Theorem \ref{tw2} with the function $F: \Gamma \to l^2(X)$ defined as $F(\Phi,z) := k(\Phi,z)\delta_{r(\Phi,z)}$, where $\delta_x: X \to \mathbb{C}$ denotes the Kronecker delta concentrated at $x \in X$.

Moreover, the above kernel satisfies (\ref{wlasnosc2}). Indeed, by (\ref{kmultiplicative},\ref{kmultiplicative2}) one has that, whenever $r(\chi) = r(\gamma)$,

\begin{align*}
K(\chi,\gamma) = k(\gamma)\overline{k(\chi)} = \overline{k(\gamma^{-1})k(\chi)} = \overline{k(\gamma^{-1}\chi)} = k(\varepsilon(\gamma^{-1}\chi))\overline{k(\gamma^{-1}\chi)} = K(\gamma^{-1}\chi,\varepsilon(r(\gamma^{-1}\chi))).
\end{align*}

On the strength the discussion following formula (\ref{wlasnosc2}), the just proven properties of the kernel $K$ mean that the latter can be used to construct a unitary `Moore--Aronszajn representation' $(\mathcal{U},\{H(K^x)\}_{x\in X})$ of $\Gamma$, with each $\mathcal{U}(\gamma): H(K^{s(\gamma)}) \rightarrow H(K^{r(\gamma)})$ satisfying
\begin{align}
\label{MArep}
\mathcal{U}(\gamma) K_{\chi} = K_{\gamma \chi}, \quad \chi \in \Gamma^{s(\gamma)}.
\end{align}
In order to better understand this representation, notice first that for any $\chi \in \Gamma$
\begin{align*}
K_\chi = k(\chi) \bar{k} \cdot \mathbf{1}_{\Gamma^{r(\chi)}},
\end{align*}
where $\mathbf{1}_{\Gamma^{r(\chi)}}$ denotes the indicator function of the fiber $\Gamma^{r(\chi)}$ (cf. Example \ref{ex:trivial}). This in particular means that for every $x \in X$ the space $H(K^x)$ is spanned by the function $\bar{k} \cdot \mathbf{1}_{\Gamma^{x}}$, which can be easily shown to be of norm $1$. Observe now that, for any chosen $\chi \in \Gamma^{s(\gamma)}$
\begin{align*}
\mathcal{U}(\gamma)(\bar{k} \cdot \mathbf{1}_{\Gamma^{s(\gamma)}}) = k(\chi)^{-1} \mathcal{U}(\gamma)K_\chi = k(\chi)^{-1} K_{\gamma \chi} = k(\chi)^{-1} k(\gamma \chi) \bar{k} \cdot \mathbf{1}_{\Gamma^{r(\gamma \chi)}} = k(\gamma) \bar{k} \cdot \mathbf{1}_{\Gamma^{r(\gamma)}},
\end{align*}
where we have used (\ref{MArep}) and (\ref{kmultiplicative}). In other words, under the above choice of the orthonormal bases of the one-dimensional spaces $H(K^x)$, the transformation $\mathcal{U}(\gamma)$ can be regarded simply as the multiplication by the complex number $k(\gamma)$.

The above groupoid $\Gamma$ together with the simple kernel given by (\ref{jadroJF1}) is by no means the only one worth investigating in the context of biholomorphisms. The construction can be generalized, e.g., by considering, for some fixed natural $N>1$, the sets $X_N:=\{(\Omega,\underline{z}) \ | \ \underline{z} \in\Omega^N\}$ and $\Gamma_N:=\{(\Phi,\underline{z}) \ | \ \underline{z} \in\Omega^N\}$, where $\underline{z}:=(z_1,\ldots,z_N)$ and the symbols $\Omega$ and $\Phi$ have the same meaning as before. For any biholomorphism $\Phi\colon\Omega_1\to\Omega_2$ we put 
\begin{align*}
& s(\Phi,\underline{z}):=(\Omega_1,\underline{z}), \qquad r(\Phi,\underline{z}):=(\Omega_2,\underline{\Phi(z)}), \qquad (\Phi,\underline{z})^{-1}:=(\Phi^{-1},\underline{\Phi(z)}),
\end{align*}
where $\underline{\Phi(z)}:=(\Phi(z_1),\ldots,\Phi(z_N))$, and we define the multiplication
\begin{align*}
(\Psi,\underline{\zeta})(\Phi,\underline{z}):=(\Psi\circ \Phi,\underline{z})
\end{align*}
provided $\Phi(\zeta) = \underline{z}$ and $\Psi: \Omega_2 \rightarrow \Omega_3$ is another biholomorphism. Finally, the embedding map is given by $\varepsilon(\Omega,\underline{z}):=(\textup{id}_{\Omega},\underline{z})$ for any $(\Omega,\underline{z}) \in X_N$. The septuple $(\Gamma_N, X_N, s, r, \varepsilon, \cdot, {}^{-1})$ is a groupoid and we can define $N$ positive definite kernels $K_1,\ldots,K_N:\Gamma_N \times \Gamma_N \to \mathbb{C}$ via
\begin{align}
\label{jadroJFk}
K_j((\Psi,\underline{\zeta}),(\Phi,\underline{z})) := \left\{ \begin{array}{cc} k(\Phi,z_j)\overline{k(\Psi,\zeta_j)} & \textup{for } r(\Phi,\underline{z}) = r(\Psi,\underline{\zeta}),
\\
0 & \textup{for } r(\Phi,\underline{z}) \neq r(\Psi,\underline{\zeta}),
\end{array} \right.
\end{align}
$j=1,\ldots,N$. Since all $K_j$'s satisfy condition (\ref{wlasnosc2}), then for any $\alpha_1,\ldots,\alpha_N>0$ and any natural exponents
$m_1,\ldots,m_N$ the function $K:=\alpha_1K_1^{m_1}+\alpha_2K_2^{m_2}+\ldots+\alpha_N K_N^{m_N}$ is a positive definite kernel\footnote{It is well known that the finite product of positive definite kernels on $A$ is itself a positive definite kernel (see, e.g., \cite[p. 6,17]{alp}). The same, of course, concerns linear combinations with positive coefficients of positive definite kernels defined on the same set $A$.} on $\Gamma_N$ also satisfying condition (\ref{wlasnosc2}). Therefore, $K$ defines a unitary `Moore--Aronszajn representation' of $\Gamma_N$, which no longer offers such a straightforward interpretation as the one presented above. This, however, goes beyond the scope of the current article and will be addressed in the future work.
\end{ex}

\section{Applications and outlook}
\label{sec:Applications}

Let us briefly discuss possible applications of the presented relationship between unitary representations of groupoids and reproducing kernels.

Kernel methods are practically utilized, e.g., in the machine learning field. Currently, their typical implementation is the classification or regression task, where the kernel-based method can be used to process the feature vector (representing the analyzed object) and produce the desired output, ensuring the minimum error, even if the data are difficult to distinguish. The most popular method is the Support Vector Machines Classifier (SVC), used to identify linearly inseparable objects. Their original features, based on which the decision is made, are transformed using the kernel function to the new space, where separation of examples belonging to various categories is easier \cite{dre}. However, one of the requirements for the kernel function $K(x,y)=\left\langle \tau(y),\tau(x)\right\rangle$ is that its input arguments are real numbers. This function can be substituted by $K(h,g) = \langle \mathcal{U}(g)v, \mathcal{U}(h)v \rangle$. As a result, the unitary representation $\mathcal{U}$  on the group (instead of object transformation $\tau$) is used.

Another domain in which kernel-based methods prove their usefulness is optimization. The problem, often encountered in data processing modules (implemented in such fields as electronics or control engineering) is the selection of the optimal kernel regarding the distance between feature vectors in the multidimensional space. According to \cite{pis}, such a measure (on the groupoids) can be used to solve the generalized version of the Traveling Salesman Problem. The overall distance to minimize is given as: 
$L=\sum_{i=1}^{N} d(c_{i},c_{i+1}),$ 
where $c_{i}$ and $c_{i+1}$ are two subsequent nodes from the graph in the optimized cycle. Using the kernel to describe distances between nodes in the new space allows for estimating the distance components as:
\begin{align*}
d(x,y)=\sqrt{K(x,x)-2K(x,y)+K(y,y)},
\end{align*}
where kernel $K$ is a real-valued function. This type of distance is described in \cite[p. 78]{gar}. In our proposed application, selection of the optimal kernel among various candidates (ensuring the minimal or maximal distance between the points) can be done without the actual transformation of the original space to the new one (which is the core of the kernel applicability in the machine learning).

Yet another possible application concerns quantum physics, where both reproducing kernels \cite{odz} and groupoids and their representations \cite{cia} have been used in the description of quantum systems. For instance, the simple representation of the pair groupoid of a two-element set, considered in Example \ref{qubitex} above, can be used in the description of a qubit --- the central notion of the theory of quantum information \cite{ibo}. In the future work we shall investigate the physical meaning and significance of the kernels and RKHS associated to the quantum-mechanically relevant unitary representations of groupoids.

On the mathematical side, let us add that the simple definition of a unitary representation of a groupoid employed in the paper can be generalized onto the setting of measurable fields of Hilbert spaces \cite{dix,pys2,ren}. The natural question whether the above-studied relationship between RKHS and unitary groupoid representations still holds in this more general setting will also be addressed in the future work. This line of research might in the end offer some new tools and insights in the fields of complex analysis, quantum physics and computer science.



\vfill\eject

\end{document}